\documentclass[a4paper]{article}
\usepackage{amsthm,amsfonts,amsmath,amssymb}
\usepackage[abbrev,nobysame]{amsrefs}
\usepackage[cp1251]{inputenc}
\usepackage[english]{babel}
\usepackage[final]{graphicx}
\usepackage{setspace}
\usepackage[12pt]{extsizes}
\begin{document}
\newcommand{\per}{{\rm per}}
\newtheorem{thm}{Theorem}
\newtheorem{lemma}{Lemma}
\newtheorem{prop}{Proposition}
\newtheorem{property}{Property}
\newtheorem{cor}{Corollary}
\newtheorem{con}{Conjecture}
\newtheorem{rem}{Remark}
\newtheorem{quest}{Question}

\author{V. N. Potapov$^1$, A. A. Taranenko$^1$, Yu. V. Tarannikov$^{2,3}$\thanks{vpotapov@math.nsc.ru, taa@math.nsc.ru, yutarann@gmail.com}}
\title{An asymptotic lower bound on the number of bent functions}
\date{\small{$^1$Sobolev Institute of Mathematics, Novosibirsk, Russia \\
$^2$Lomonosov Moscow State University, Moscow, Russia \\
$^3$Moscow Center of Fundamental and Applied Mathematics, Moscow, Russia}}

\maketitle

\begin{abstract}
A Boolean function $f$ on $n$ variables is said to be a bent function if the absolute value of all its Walsh coefficients is $2^{n/2}$. Our main result is a new asymptotic lower bound on the number of Boolean bent functions.  It is based on a modification of the Maiorana--McFarland family of bent functions and recent progress in the estimation of the number of transversals in latin squares and hypercubes.  By-products of our proofs are the asymptotics of the logarithm of the numbers of partitions of the Boolean hypercube into $2$-dimensional affine and linear subspaces.

\end{abstract}

\section{Definitions, preliminaries, and main results}

Boolean functions and, in particularly, bent functions are widely used in cryptography~\cite{carle.boolean,mesnager.bentfunc,tokareva.bentfunc}, and problems of their existence and enumeration are important. Asymptotic bounds on the numbers of certain cryptographic functions were obtained, for example, in~\cite{CGGMR.corimm,pot.medcorimm,tokareva.iteratbentfunc}.

The literature on bent functions is usually devoted to their existence and constructions  whereas it does not pay much attention to bounds on cardinalities of classes of bent functions. The most known general bounds  are
the cardinality of the Majorana--McFarland family as well as a cumbersome Agievich's formula~\cite{agievich.repbent}. There are also enumerative results for bent functions on small number of variables $n$ (see, e.g.,~\cite{LanLean.benton8} for bent function on $8$ variables). Meanwhile, in this paper we are interested in asymptotic bounds and do not compare them with exact results for small $n$.  One of the problems with estimations for small $n$  is that an  affine equivalence still gives a substantial contribution to the number of bent functions.

In Table 1 we present without proof our analysis of the logarithms
of cardinalities for some relatively rich classes of bent functions on $n$ variables. It is well known that bent functions exist if and only if $n$ is even.

\vskip 5truemm

\centerline{
\vbox{\offinterlineskip
\hrule
\halign{&\vrule#&\,\hfil#\hfil\,\cr
height5pt&\omit&&\omit&&\omit&\cr
&\hbox{Class and reference}&&\hbox{Asymptotics of $\log_2$ of cardinality}&\cr
height5pt&\omit&&\omit&&\omit&\cr
\noalign{\hrule}
height2pt&\omit&&\omit&&\omit&\cr
&\strut MM family~\cite{mcfarland.difset}&&$\log_2 |{\cal M}(n)|={n\over 2}\cdot 2^{n/2}(1+o(1))$  &\cr
height2pt&\omit&&\omit&&\omit&\cr
\noalign{\hrule}
height2pt&\omit&&\omit&&\omit&\cr
&\strut completed MM family~\cite{mesnager.bentfunc}&&$\log_2 |{\cal M}^{\#}(n)|={n\over 2}\cdot 2^{n/2}(1+o(1))$ &\cr
height2pt&\omit&&\omit&&\omit&\cr
\noalign{\hrule}
height2pt&\omit&&\omit&&\omit&\cr
&\strut iterative construction CC~\cite{CantChap.decomposing, carle.boolean}&&$\log_2 |CC(n)| \leq  3^{n/2}(1+o(1))$ &\cr
height2pt&\omit&&\omit&&\omit&\cr
\noalign{\hrule}
height2pt&\omit&&\omit&&\omit&\cr
&\strut special subclass of CC~\cite{tokareva.iteratbentfunc}&&$\log_2 |CC'(n)|= 2^{n/2}(1+o(1))$ &\cr
height2pt&\omit&&\omit&&\omit&\cr
\noalign{\hrule}
height2pt&\omit&&\omit&&\omit&\cr
&\strut ${\cal C}$ class~\cite{carle.newbent}&&$\log_2 |{\cal C}(n)|={n\over 2}\cdot 2^{n/2}(1+o(1))$ &\cr
height2pt&\omit&&\omit&&\omit&\cr
\noalign{\hrule}
height2pt&\omit&&\omit&&\omit&\cr
&\strut ${\cal D}$ class~\cite{carle.newbent}&&$\log_2 |{\cal D}(n)|={n\over 2}\cdot 2^{n/2}(1+o(1))$  &\cr
height2pt&\omit&&\omit&&\omit&\cr
\noalign{\hrule}
height2pt&\omit&&\omit&&\omit&\cr
&\strut Agievich bound~\cite{agievich.repbent}&&$\log_2 A(n)={n\over 2}\cdot 2^{n/2}(1+o(1))$ &\cr
height2pt&\omit&&\omit&&\omit&\cr
\noalign{\hrule}
height2pt&\omit&&\omit&&\omit&\cr
&\strut special subclass of ${\cal PS}$~\cite{mesnager.bentfunc}&&$\log_2 |{\cal PS}_{ap}(n)|=2^{n/2}(1+o(1))$  &\cr
height2pt&\omit&&\omit&&\omit&\cr
\noalign{\hrule}
height2pt&\omit&&\omit&&\omit&\cr
&\strut ${\cal P}{\rm artial}\,\, {\cal S}{\rm pread}$ family~\cite{mesnager.bentfunc}&&$\log_2 |{\cal PS}(n)|\le {n^2\over 8}\cdot 2^{n/2}(1+o(1))$  &\cr
height2pt&\omit&&\omit&&\omit&\cr
\noalign{\hrule}
height2pt&\omit&&\omit&&\omit&\cr
&\strut Construction (K)~\cite{agievich.bentrec,BakTar.onebentfunc,CesMei.bentfromplat}&&$\log_2 |K_{(n/2)-k}(n)|\le {(2k+1) n\over 2^{k+1}}\cdot 2^{n/2}(1+o(1))$  &\cr
height2pt&\omit&&\omit&&\omit&\cr
\noalign{\hrule}
height2pt&\omit&&\omit&&\omit&\cr
&\strut &&   &\cr
height2pt&\omit&&\omit&&\omit&\cr
\noalign{\hrule}
height2pt&\omit&&\omit&&\omit&\cr
&\strut Construction (K)~\cite{agievich.bentrec,BakTar.onebentfunc,CesMei.bentfromplat}&&$\log_2 |K_{(n/2)-1}(n)|={3n\over 4}\cdot 2^{n/2}(1+o(1))$ &\cr
height2pt&\omit&&\omit&&\omit&\cr
}
\hrule}
}

 \vskip 5truemm
\begin{center}
Table 1. Cardinalities of some classes of bent function on $n$ variables.
\end{center}

The upper bounds on the number of functions given by the Partial Spread or iterative  constructions are obtained by the means of rough estimations of their components, and, in fact, they may be quite far from true values. The main result of the present paper is the asymptotics in the last row of Table 1. Here $K_{(n/2)-k}(n)$ stands for the class of bent functions on $n$ variables given by the construction~(K) with $n_1 =(n/2)-k$.

Our lower bound is given by a class~(K) of bent functions proposed in~\cite{BakTar.onebentfunc} that is a variance of a construction from~\cite{agievich.bentrec}. A similar construction of bent functions over finite fields of characteristic $p \neq 2$ was also proposed in~\cite{CesMei.bentfromplat}. Moreover, in~\cite{carle.extMaiMcF} and~\cite{agievich.repbent} it was considered an analog of  the construction~(K) that uses axis-aligned planes instead of affine ones. In~\cite[Theorem 4, 5]{CesMeiPott.generalMM} it is proved that the ternary analog of the construction~(K) gives bent functions outside the completed Maiorana--McFarland family, as well as it does not produce all bent functions. 
 
We show that the construction (K) gives more bent functions than any other known explicit construction.  Nevertheless, it is not improbable that some of the above general methods produce even more bent functions, when their parameters are appropriately chosen and estimated.

Let $F_2 = \{ 0,1\}$. The set $F_2^n$ is called the \textit{$n$-dimensional Boolean hypercube} (or the \textit{Boolean $n$-cube}). The hypercube $F_2^n$  equipped with scalar multiplication and coordinate-wise  modulo $2$ addition $\oplus$ is an $n$-dimensional vector space.  Its zero element is $\overline{0} = (0, \ldots, 0)$. A set $C \subseteq F_2^n$ is called a \textit{$k$-dimensional affine subspace} if $C = a \oplus S$ for some $a \in F_2^n$ and a $k$-dimensional linear subspace $S$ of $F_2^n$.

For $x,y \in F_2^n$, $x= (x_1, \ldots, x_n)$, $y = (y_1, \ldots, y_n)$, we define their \textit{inner product} as
$$\langle x,y \rangle = x_1 y_1 \oplus \cdots \oplus x_n y_n. $$

A function $f: F_2^n \rightarrow F_2$ is said to be a \textit{Boolean function} on $n$ variables.

The \textit{Walsh transform} of a Boolean function $f$ is a function $W_f: F_2^n \rightarrow \mathbb{Z}$  such that
$$W_f(u) = \sum\limits_{x \in F_2^n} (-1)^{ \langle u,x \rangle \oplus f(x)}.$$
The values $W_f(u)$ are called \textit{Walsh coefficients}, and the set of all Walsh coefficients is called the \textit{Walsh spectrum} of $f$. The \textit{support} of the Walsh spectrum is the set  $\{u:  W_f(u) \neq 0 \}$.

A Boolean function $f$ on $n$ variables is said to be a \textit{bent function} if the Walsh spectrum of $f$ consists of $\pm 2^{n/2}$, and $f$ is a \textit{plateaued function} if all its Walsh coefficients  are equal to $\pm 2^{k}$ or $0$, for some integer $k$.  We use $b_n$ to denote the number of  bent functions on $n$ variables.

Bent functions $f$ and $g$ are \textit{affinely equivalent}, if there is a nondegenerate binary matrix $L$ of size $n \times n$  and $a \in F_2^n$ such that
$$g(x) = f(Lx \oplus a).$$ For each bent function $f$ there are no more than $2^{n^2 +n}$ affinely equivalent bent functions.

It is well known (see \cite{carle.boolean,CarMes.bentfucres,mesnager.bentfunc}) that the algebraic degree (the degree of the Zhegalkin polynomial which is also known as the algebraic normal form) of
a bent function $f$ on $n$ variables is at most  $n/2$. Therefore, the number $b_n$ of bent functions is not greater than $ 2^{\sum\limits_{i=0}^{n/2} {n \choose i}}$, and, consequently, $\log_2 b_n\leq 2^{n-1}+ \frac{1}{2} {n \choose n/2} $. In \cite{CarKlap.resfunc} and \cite{agievich.contbent} there are slightly better upper bounds on the number of bent functions, but asymptotically both of them are $\log_2
b_n\leq 2^{n-1}(1+o(1))$.   In~\cite{potapov.upbent} the following improvement of the upper bound is stated.

\begin{thm}[\cite{potapov.upbent}] \label{upperbpund}
Let $b_n$ be the number of bent functions on $n$ variables. Then it is not greater than $6^{3\cdot2^{n-6}}2^{\cdot2^{n-2}(1+o(1))}$ as $n\rightarrow\infty$. In particular,
$$\log_2 b_n\leq 3\cdot2^{n-3}(1+o(1)).$$
\end{thm}

Note that Tokareva's conjecture \cite{tokareva.iteratbentfunc} on the decomposition of
Boolean functions into a sum of bent functions suggests that $\log_2
b_n\geq 2^{n-2}+\frac{1}{2}{n \choose n/2}$.

Till now the class of  Maiorana--McFarland functions~\cite{mcfarland.difset} was considered as  the richest family of bent functions (up to some extensions). This class consists of functions of the form
$$f (x,y) = f(x_1, \ldots, x_m, y_1, \ldots, y_m) = \psi (y) \bigoplus\limits_{i=1}^m x_i \pi_i (y)$$
and functions that are affinely equivalent to them. Here $n = 2m$, $\psi (y)$ is an  arbitrary Boolean function on $m$ variables, and $\pi$ is an arbitrary permutation of  $F_2^m$, $\pi(y) = (\pi_1 (y), \ldots, \pi_m(y))$.

The choice of permutation $\pi$ and Boolean function $\psi$ contributes $2^{n/2}! \cdot 2^{2^{n/2}}$ bent functions to the  Maiorana--McFarland family.  Taking into account affinely equivalent functions, we see that the completed Maiorana--McFarland family contains no more than $2^{n/2}! \cdot 2^{2^{n/2}}  \cdot 2^{n^2 + n}$ bent functions.

Using the Stirling's approximation,
\begin{equation} \label{stirling}
\log_2 N! = N \log_2 N - N \log_2 e + o(N),
\end{equation}
we conclude that the logarithm of the number $b_n$ of bent functions on $n$ variables satisfies
 $$\log_2 b_n \geq \frac{n}{2}  \cdot 2^{n/2} +(1 - \log_2 e )  \cdot   2^{n/2} + o(2^{n/2}).$$

The main result of the present paper is the following asymptotic bound on the number of bent functions.

\begin{thm} \label{mainbound}
Let $b_n$ be the number of bent functions on $n$ variables, where $n$ is even. Then
 $$\log_2 b_n \geq \frac{3n}{4}  \cdot 2^{n/2} - 2 \log_2 e  \cdot 2^{n/2} + o(2^{n/2}).$$
\end{thm}

For additional information on bent functions and their number the reader is reffered to papers and monographs~\cite{CarMes.bentfucres,mesnager.bentfunc,tokareva.bentfunc,tokareva.iteratbentfunc}.

\section{Construction of bent functions} \label{bentconst}

For a Boolean variable $a \in F_2$, we use a notation $a^1=a$ and $a^0=a\oplus1$. In particular, $a^b=1\Leftrightarrow a=b$. Moreover, for $x,y\in F^n_2$,  $x= (x_1, \ldots, x_n)$, $y = (y_1, \ldots, y_n)$,  we define $x^y=x_1^{y_1}\cdots x_n^{y_n}$.

Let us describe a family of Boolean functions that gives the lower bound in Theorem~\ref{mainbound}.

\textbf{Construction (K):}  Let  $n = n_1 + n_2$, $n_2 \geq n_1$, $n$ and $n_2 - n_1$ be even. Assume that $\{ C_a \}_{a \in F_2^{n_1}}$, $C_a \subseteq F_2^{n_2}$ is an ordered partition of $F_2^{n_2}$ into $2^{n_1}$ affine subspaces of dimensions $n_2 - n_1$.  Define a Boolean function $f$ on $n$ variables as
$$f (x,y) = \bigoplus\limits_{{a}\in F_2^{n_1}} f_{a}(y){x}^{a},$$
where $x \in F_2^{n_1}, y \in F_2^{n_2}$, and  $f_a$ are plateaued functions such that the support of the Walsh spectrum of $f_a$ is exactly $C_a$.

Note that the choice of different ordered partitions of  $F_2^{n_2}$ into $(n_2 - n_1)$-dimensional affine subspaces or different plateaued functions $f_a$ for the same partition results in different functions $f$.

To prove that such a function $f$ is well defined, it is sufficient to construct a plateaued function $g$ such that the support of its Walsh spectrum is equal to any given affine subspace $C$ of even dimension.  This fact was previously established in~\cite{taran.specplat}, but we prove it here for the sake of completeness.

For this purpose, we need the following properties of Walsh coefficients. They can be found, e.g., in books~\cite{mesnager.bentfunc,tokareva.bentfunc} or can be derived directly from the definitions.

\begin{prop} \label{Walshprop}
Let $f$ be a Boolean function on $n$ variables.
\begin{enumerate}
\item Suppose that  $f$ has exactly $k$ essential variables: $f(x,y)=g(x)$, where $x \in F_2^{k}$, $y \in F_2^{n-k} $, and  $g$ is a Boolean function on $k$ variables. Then for all $u \in F_2^{k}$ we have $W_f(u,\overline{0})= 2^{n-k} W_g(u)$, and $W_f(u,v)=0$ if $v \in F_2^{n-k} \setminus \{ \overline{0} \}$.

\item Let $f(x)=g(Lx)$ for some nondegenerate binary matrix $L$ of sizes $n \times n$ and a Boolean function $g$ on $n$ variables. Then $W_f(u)= W_g((L^{-1})^T u)$.

\item Assume that $f(x)=g(x) \oplus \langle a,x \rangle$ for some $a\in F^n_2$ and a Boolean function $g$ on $n$ variables.
Then $W_f(u) =  W_g(u\oplus a)$.
\end{enumerate}
\end{prop}

\begin{prop}[\cite{taran.specplat}] \label{bentplatau}
Let $C\subseteq F^n_2$ be an affine subspace of even dimension $k$. There exists a one-to-one correspondence  between plateaued functions $f$ on $n$ variables, whose support of the Walsh spectrum is equal to $C$, and bent functions $g$ on $k$ variables. Moreover, the absolute value of all nonzero Walsh coefficients $W_f$ of the function $f$ is $2^{n - k/2}$.
\end{prop}

\begin{proof}
Let $g$ be a bent function on $k$ variables. By the definition, $|W_g (u)| = 2^{k/2}$ for all $u \in F_2^k$.  Using Proposition~\ref{Walshprop}(1), we construct a Boolean function $h$ on $n$ variables all of whose nonzero Walsh coefficients are equal to $\pm 2^{n - k/2}$ and located in a $k$-dimensional subcube of the Boolean $n$-cube. With the help of Proposition~\ref{Walshprop}(2) and (3), we put the support of the Walsh spectrum of $h$ to the affine subspace $C$ and obtain the desired plateaued function $f$. Note that the absolute value of all nonzero Walsh coefficients of $f$ is $2^{n - k/2}$.

Reversing this reasoning, we have the equivalence.
\end{proof}

\textbf{Example.} Let $n = 6$, $n_1 = 2$, $n_2 = 4$. Consider the partition $\{ C_{00}, C_{01}, C_{10}, C_{11}\}$ of $F_2^4$ into  $2$-dimensional affine subspaces:
\begin{gather*}
C_{00} = \{  0000, 0100, 1010, 1110\};~~ C_{01} = \{0010, 0110, 0011, 0111 \}; \\
C_{10} = \{ 1000, 1100, 1001, 1101\}; ~~C_{11} = \{ 0001, 0101, 1011, 1111\}.
\end{gather*}

Using Propositions~\ref{Walshprop}, \ref{bentplatau} and the bent function $ y_1 y_2$, we construct the following  plateaued functions $f_a$, whose support of the Walsh spectrum is exactly $C_a$:
$$
\begin{array}{ll}
f_{00} = y_2(y_1 \oplus y_3); & f_{01} = y_2 y_4 \oplus y_3; \\
f_{10} = y_2y_4 \oplus y_1; & f_{11} = y_2(y_1 \oplus y_3) \oplus y_4.
\end{array}
$$

So the function $f$ given by the construction (K) is
$$f(x,y) = (y_1\oplus y_3 \oplus y_4) (x_1 x_2 \oplus y_2 x_1 \oplus y_2 x_2)  \oplus  y_3 x_2 \oplus y_1 x_1 \oplus y_2(y_1\oplus y_3).$$

Let us prove now that the construction (K) produces only bent functions.

\begin{thm}[\cite{BakTar.onebentfunc}] \label{Kisbent}
Every function $f$ given by the construction (K) is bent.
\end{thm}

\begin{proof}
By the construction,
$$f=\bigoplus\limits_{a\in F_2^{n_1}} f_{a}(y){x}^{a},$$
where $f_{a}:F^{n_2}_2\rightarrow F_2$ are plateaued functions whose supports of Walsh spectra equal $C_a$ and  $\{ C_a \}_{a \in F_2^{n_1}}$ is an ordered partition of $F^{n_2}_2$ into $(n_2-n_1)$-dimensional affine subspaces.

By Proposition~\ref{bentplatau}, the absolute value of all nonzero Walsh coefficients $W_{f_a}$ of the plateaued functions $f_a$ is $2^{n_2 - \frac{n_2 - n_1}{2}} = 2^{n/2}$.

Using the definition of Walsh coefficients and the function $f$, we have
$$W_f(u,v)=\sum\limits_{x\in F_2^{n_1}} \sum\limits_{y\in F_2^{n_2}}
(-1)^{\langle u,x\rangle \oplus\langle v,y\rangle \oplus\bigoplus\limits_{{a}\in F_2^{n_1}}
f_{{a}}(y){x}^{a}},$$
where  $u \in F_2^{n_1}$ and $v \in F_2^{n_2}$.

From properties of the operation $x^a$, we see that for given $x \in F_2^{n_1}$ the sum $\bigoplus\limits_{{a}\in F_2^{n_1}}
f_{{a}}(y){x}^{a}$ has the unique nonzero summand $f_x(y)$ (when $a = x$). Therefore,
\begin{gather*}
W_f(u,v)=\sum\limits_{x\in F_2^{n_1}} \sum\limits_{y\in F_2^{n_2}} (-1)^{\langle u,x\rangle \oplus\langle v,y\rangle \oplus f_x(y)}
=\sum\limits_{x\in F_2^{n_1}}(-1)^{\langle u,x\rangle}\left( \sum\limits_{y\in F_2^{n_2}}(-1)^{\langle v,y\rangle\oplus f_{x}(y)}\right) \\
=\sum\limits_{x\in F_2^{n_1}}(-1)^{\langle u,x\rangle}W_{f_x}(v).
\end{gather*}

Note that for given  $v \in C_b$ the sum $\sum\limits_{x\in F_2^{n_1}}(-1)^{\langle u,x\rangle}W_{f_x}(v)$ has the unique nonzero summand  $(-1)^{\langle u,b\rangle}W_{f_b}(v)$, since all other plateaued functions $f_a$ do not have $v$ in the support of their Walsh spectrum.  Thus $|W_f(u,v)|= |W_{f_b}(v) |= 2^{n/2}$ for all $u \in F_2^{n_1}$ and $v \in F_2^{n_2}$. It means that $f$ is a bent function.
\end{proof}

In what follows, we denote by $\widetilde{N}_{m}^k$ the number of ordered partitions of the space $F_2^{m}$ into $k$-dimensional affine subspaces and by $N_{m}^k$ the number of all such unordered partitions. (A partition of $F_2^{m}$ into affine subspaces is unordered when the order of subspaces in the partition does not matter). Proposition~\ref{bentplatau} and Theorem~\ref{Kisbent} easily imply the following formula for the number of bent functions in the construction (K).

\begin{thm} \label{numberbent}
Let $n = n_1 + n_2$, $n_2 \geq n_1$, and $n$ and $n_2 - n_1$ be even. Then the number $B_n$ of bent functions given by the construction (K) is
\begin{equation}\label{number_from_K} 
B_n = (b_{n_2-n_1})^{2^{n_1}}  \cdot \widetilde{N}_{n_2}^{n_2-n_1},
\end{equation}
where $b_{n_2 - n_1}$ is the number of bent functions over $n_2 - n_1$ variables, $\widetilde{N}_{n_2}^{n_2-n_1}$ is the number of ordered partitions of the space $F_2^{n_2}$ into $2^{n_1}$ affine subspaces of dimensions $n_2 - n_1$.
\end{thm}

Note that there are other generalizations of the construction~(K).  For example, instead of  ordered partitions of $F_2^{n_2}$ into affine subspaces of dimensions $n_2 - n_1$, one can also use ordered partitions of $F_2^{n_2}$  into appropriate sets $C_a$ of cardinality $2^{n_2 - n_1}$ and look for plateaued functions $f_a$ whose spectrum supports are $C_a$.

Although this generalization obviously gives more bent functions, the determination of the  exact asymptotics of their number is quite hard. Indeed, firstly we need to effectively enumerate  partitions  of $F_2^{n_2}$  into sets $C_a$ of cardinality $2^{n_2 - n_1}$ that can have a complicated structure.   Next, for every set $C_a$ used in the partition we should show that there exists a plateaued function $f$ with the support of the Walsh spectrum equal to $C_a$. More than that, we need good lower bounds on the numbers of such plateaued functions.  

Note that the exact values of the numbers of plateaued functions with a given  support of the Walsh spectrum  are known only for narrow families~\cite{khalyavin.spectrum}, and these numbers prove to be much smaller than the number of bent functions with the same cardinality of the support of the Walsh spectrum. Thus, their contribution to the total number is expected to be much less than the impact of affine subspaces.

In~\cite{Hodzic.designing} Hod\v{z}i\'{c} et al. provide examples of partitions of Boolean hypercubes into appropriate sets and construct a family of bent functions with the help of such generalization of the construction~(K). Unfortunately,  the number of  these bent functions turns out to be relatively small.  

Thus we will focus on what is more realistic to be accomplished.   More exactly,  we estimate the value $\widetilde{N}_{n_2}^{n_2-n_1}$ and obtain  a new asymptotic lower bound on the logarithm of the number of Boolean bent functions using the formula (\ref{number_from_K}).

\section{Proof of the lower bound} \label{mainproof}

The key element of the proof of Theorem~\ref{mainbound} is an estimation of the number of ordered partitions of $F_2^m$ into $2$-dimensional affine subspaces. Meanwhile, here we establish the asymptotics of the logarithm of the number of  the unordered ones. For shortness, we denote  the number of unordered partitions of $F_2^m$ into $2$-dimensional affine subspaces by $N_m$ ($N_m = N^2_{m}$).

\begin{thm} \label{affindecomp}
The number $N_m$ of all  unordered partitions of $F_2^m$ into $2$-di\-men\-sio\-nal affine subspaces satisfies
$$ \frac{m}{2}  \cdot 2^m + c_1 \cdot 2^{m} + o(2^m)   \leq  \log_2 N_m \leq \frac{m}{2}  \cdot 2^m +  c_2  \cdot 2^{m} + o(2^m),$$
where $c_1 = -1 - \frac{3}{4} \log_2 e  \approx -2.08$, $c_2  = \frac{7}{16}  - \frac{11}{16} \log_2 3 \approx - 0.65$.
\end{thm}

It is easy to see that the numbers of ordered and unordered partitions  of $F_2^m$ into $k$-dimensional affine subspaces are connected in the following way.

\begin{prop} \label{ordandunord}
If $\widetilde{N}_{m}^k$ is the number of ordered partitions of the space $F_2^{m}$ into $k$-dimensional affine subspaces and $N_{m}^k$ is the number of unordered ones, then
$$\widetilde{N}_{m}^k = 2^{m-k} ! \cdot N^k_{m}.$$
\end{prop}

The proof of Theorem~\ref{affindecomp} needs more definitions and some auxiliary results on latin hypercubes, their transversals, and perfect matchings in hypergraphs.

A \textit{$d$-dimensional latin hypercube of order $n$} is a $d$-dimensional matrix $Q = ( q_{\alpha} )$ of order $n$ whose entries indexed by $\alpha = (\alpha_1, \ldots, \alpha_d)$, $\alpha_i \in \{ 1, \ldots, n\}$, where $Q$ is filled by $n$ symbols so that each symbol appears in each line (1-dimensional submatrix) exactly once. A \textit{transversal} in a latin hypercube is a collection of $n$ entries hitting each hyperplane ($(d-1)$-dimensional submatrix) and each symbol exactly once.

Actually, we are interested in transversals in specific latin hypercubes.
Let  $Q_m$ be  the $3$-dimensional latin hypercube of order $2^m$ corresponding to the Cayley table of the iterated group $\mathbb{Z}_2^m$. In more details, its entry $q_{\alpha_1, \alpha_2, \alpha_3} = \alpha_4$, $\alpha_{i} \in F_2^m$, if and only if $\alpha_1 \oplus \cdots \oplus \alpha_4 = \overline{0}$. For more information about latin hypercubes, their transversals, and iterated groups and quasigroups, see e.g.~\cite{my.iter}.

In what follows, instead of entries of $Q_m$ we consider tuples $(\alpha_1, \ldots, \alpha_4)$, $\alpha_{i} \in F_2^m$ satisfying $\alpha_1 \oplus \cdots \oplus \alpha_4 = \overline{0}$. Such a notation comprises the index and the value of an entry of the latin hypercube. Then a transversal in  the latin hypercube $Q_m$ is a collection of $2^m$ tuples
$$(\alpha_1^1, \ldots, \alpha_4^1), \ldots, (\alpha^{2^m}_1, \ldots, \alpha_4^{2^m}) $$
such that for each $j=1, \ldots, 4$ all $\alpha^i_j$ are different, $i = 1, \ldots, 2^m$.

In~\cite[Theorem 7.2]{eberhard.moreaddtrip} it was found the asymptotics of the number of transversals in iterated abelian groups.  In particular, we have the following estimation of the number of transversals in $Q_m$.

\begin{thm}[\cite{eberhard.moreaddtrip}] \label{eberbound}
The number $T_m$ of transversals in the $3$-dimensional latin hypercube $Q_m$  of order $2^m$ that is the Cayley table of the iterated group $\mathbb{Z}_2^m$ is
$$T_m = (1 + o(1)) \frac{2^m!^3}{2^{m  ( 2^m -1)}}$$
as $m \rightarrow \infty$.
\end{thm}

There is a connection between the number of unordered partitions of $F_2^m$ into $2$-dimensional affine spaces and the number of transversals in  $Q_{m}$.

\begin{prop} \label{latindecomp}
The number $N_m$ of unordered partitions of $F_2^m$ into $2$-di\-men\-sio\-nal affine subspaces is not less than the number of transversals in the latin hypercube $Q_{m-2}$:
$$N_m \geq T_{m-2}.$$
\end{prop}

\begin{proof}
For shortness, we use a notation $M = 2^{m-2}$. Let a collection $R$ of $M$ tuples
$$(\alpha_1^1, \ldots, \alpha_4^1), \ldots, (\alpha^M_1, \ldots, \alpha_4^M) $$
be a transversal in the latin hypercube $Q_{m-2}$. Recall that  $\alpha_{j}^i \in F_2^{m-2}$,  $\alpha_1^i \oplus \cdots \oplus \alpha_4^i = \overline{0}$  for all $i$, and for a fixed $j$ all $\alpha^i_j$ are different.

To each such collection $R$ we put in correspondence a collection $R'$ of $M$ sets
$$\{ \beta_1^1, \ldots, \beta_4^1\}, \ldots, \{ \beta^M_1, \ldots, \beta_4^M\}, $$
where  $\beta^i_j \in F_2^m$ and
$$\beta^i_j =\left\{   \begin{array}{l}
(\alpha^i_j, 0, 0) \mbox{ if } j = 1; \\
(\alpha^i_j, 0, 1) \mbox{ if } j = 2; \\
(\alpha^i_j, 1, 0) \mbox{ if } j = 3; \\
(\alpha^i_j, 1, 1) \mbox{ if } j = 4.
 \end{array}  \right. $$

Let us show that $R'$ is an unordered partition of $F_2^m$ into $2$-dimensional affine subspaces.

First of all,  we still have  $\beta_1^i \oplus \cdots \oplus \beta_4^i = \overline{0}$  for all $i$. It means that each set $\{ \beta_1^i, \ldots, \beta_4^i\}$ is a $2$-dimensional affine subspace in $F_2^m$. Note that different tuples $(\alpha_1^i, \ldots, \alpha_4^i)$ correspond to different $2$-dimensional affine subspaces $\{ \beta_1^i, \ldots, \beta_4^i\}$.

Since $R$ is a transversal in the latin hypercube $Q_{m-2}$, for given $\alpha \in F_2^{m-2}$ and $j \in \{ 1, \ldots, 4 \}$ there is a unique $i \in \{1, \ldots, M \}$ such that $\alpha $ coincides with some $\alpha_i^j$ from the collection $R$. So by the construction, for each $\beta \in F_2^m$ there is a unique $\beta^i_j$ from the collection $R'$ such that  $\beta = \beta_i^j$. Since each set in $R'$ has all different elements, we conclude that $R'$ is a partition of $F_2^m$.

Thus, $R'$ is an unordered partition of $F_2^m$ into $2$-dimensional affine subspaces. Since different transversals $R$ in $Q_{m-2}$ consist of different collections of tuples, they give different partitions $R'$.
\end{proof}

To prove the upper bound in Theorem~\ref{affindecomp}, one can use bounds on the number of perfect matchings in an appropriate hypergraph.

Let $H(X,W)$ be a hypergraph with the vertex set $X$ and a hyperedge set $W$. A hypergraph $H$ is said to be \textit{$d$-uniform} if each hyperedge consists of exactly $d$ vertices and \textit{$k$-regular} if each vertex appears in exactly $k$ hyperedges.

A \textit{perfect matching} in a hypergraph $H$ is a collection of hyperedges that cover each vertex of a hypergraph exactly once. Let $PM(H)$ denote the number of perfect matchings in $H$.

Consider a hypergraph $\mathcal{H}_m$, whose vertex set $V(\mathcal{H}_m)$ is the set $F_2^m$ and the hyperedge set $ W(\mathcal{H}_m) $ is the set of all affine subspaces in $F_2^m$:
$$ \{ x_1, \ldots, x_4 \} \in W(\mathcal{H}_m) \Leftrightarrow x_1 \oplus \cdots \oplus x_4 = \overline{0}.$$

It is easy to see that $\mathcal{H}_m$ is a $4$-uniform $k$-regular hypergraph on $2^m$ vertices, where $k = \frac{1}{6}(2^m-1) (2^m -2)$. Moreover, the number $N_m$ of unordered partitions of $F_2^m$ is exactly the number of perfect matchings in $\mathcal{H}_m$.

From~\cite[Corollary 2]{my.1fact} we have the following upper bound on the number of perfect matchings in uniform regular hypergraphs.

\begin{thm}[\cite{my.1fact}] \label{factbound}
Let $H$ be a $d$-uniform $k$-regular hypergraph on $n$ vertices, $d \geq 3$. Then the number $PM(H)$ of perfect matchings in $H$ satisfies
$$PM(H) \leq (  \mu \cdot k )^{n/d},$$
where $\mu = \mu(d) = \frac{d^d d!^{1/d}}{d!^2}$ for $d \geq 4$ and $\mu = \frac{3}{2^{2/3}}$ for $d = 3$.
\end{thm}

Now we are ready to find the asymptotics of the logarithm of the number of unordered partitions of $F_2^m$ into $2$-dimensional affine subspaces.

\begin{proof}[Proof of Theorem~\ref{affindecomp}]

We start with the proof of the lower bound.
By Proposition~\ref{latindecomp}, the number $N_m$ of unordered partitions of $F_2^m$ into $2$-dimensional affine subspaces is not less than the number of transversals in the latin hypercube $Q_{m-2}$:
$$N_m \geq T_{m-2}.$$

By Theorem~\ref{eberbound}, we have that
$$T_{m-2} = (1 + o(1)) \frac{2^{(m-2)}!^3}{2^{(m-2) \cdot  (2^{m-2} - 1) }} \mbox{ as } m \rightarrow \infty.$$
Using the Stirling's approximation~(\ref{stirling}), we deduce
\begin{gather*}
\log_2 N_m \geq \log_2 T_{m-2}
= \frac{m}{2}  \cdot 2^m - \left(1 + \frac{3}{4} \log_2 e \right) \cdot 2^{m} + o(2^m).
\end{gather*}

For the upper bound we use Theorem~\ref{factbound} and the fact that $N_m$ is the number of perfect matchings in the hypergraph $\mathcal{H}_m$:

$$N_m \leq  \left(  \frac{\mu}{6} \cdot (2^m - 1)(2^m - 2) \right)^{2^{m-2}}. $$

$$\log_2 N_m \leq  \frac{m}{2}  \cdot 2^m +  \frac{1}{4} \log_2 \frac{\mu}{6} \cdot 2^{m} + o(2^m).$$

Since $\mu =   \frac{4^4 \cdot  4!^{1/4}}{4!^2}$, we have $\log_2 \frac{\mu}{6}= \frac{7}{4} - \frac{11}{4} \log_2 3$.
\end{proof}

\begin{rem}
 In~\cite{luria.nqueens} it was announced another upper bound on the number of perfect matchings with $\mu = e^{-3}$ in the case of hypergraph $\mathcal{H}_m$. If we use it instead of Theorem~\ref{factbound}, then we obtain
$$ \log N_m \leq \frac{m}{2}  \cdot 2^m +  c'_2  \cdot 2^{m} + o(2^m),$$
where $c'_2  = -\frac{1}{4} ( 1 +  \log_2 3 + 3 \log_2 e ) \approx - 1.73$.
\end{rem}

At last, let us prove the lower bound on the number of bent functions.

\begin{proof}[Proof of Theorem~\ref{mainbound}]

Let $n$ be even,  $n_1 = n/2 - 1$, $n_2 = n/2 + 1$.

By Theorem~\ref{numberbent} and Proposition~\ref{ordandunord}, the number of bent functions given by the construction (K) for these $n_1$ and $n_2$ is
$$ B_n = 2^{3 \cdot 2^{n/2 -1}} \cdot 2^{n/2-1}! \cdot  N_{n/2+1},$$
since there are $8 = 2^3$ bent functions on $2$ variables.

Using Theorem~\ref{affindecomp} and the Stirling's approximation~(\ref{stirling}), we get
\begin{gather*}
\log_2 B_n  \geq    \frac{3n}{4} \cdot 2^{n/2} - 2\log_2 e \cdot  2^{n/2} + o (2^{n/2}).
\end{gather*}
\end{proof}

\begin{prop}
The asymptotically maximal number of bent functions given by  the construction (K) for fixed $n_2 - n_1$ is achieved when $n_1 = n/2 -1$, $n_2 = n/2 +1$.
\end{prop}

\begin{proof}

If $n_1 = n_2 = n/2$, then the construction (K) coincides with the Maiorana--McFarland family of bent functions, whose number is smaller than one from Theorem~\ref{mainbound}.

Let $n_1 = n/2 - k$, $n_2 = n/2 +k$, where $k \in \mathbb{N}$, $k \geq 1$ is fixed. A fundamental contribution to the number of such bent functions is given by  the number $\widetilde{N}_{n_2}^{2k}$ of ordered partitions of $F_2^{n_2}$ into $2k$-dimensional affine subspaces. For this purpose we again use a connection to perfect matchings in a special hypergraph and  Theorem~\ref{factbound}.

Let $\mathcal{H}_{n_2}^k$ be $2^{2k}$-uniform hypergraph on $2^{n_2}$ vertices, where each hyperedge is a $2k$-dimensional affine subspace in $2^{n_2}$. The degree of this hypergraph (number of $2k$-dimensional affine subspaces containing a given $x \in F_2^{n_2}$) is not greater than  $2^{2kn_2}$. By Theorem~\ref{factbound} and Proposition~\ref{ordandunord},
$$\widetilde{N}_{n_2}^{2k} \leq 2^{n/2 - k}!  \cdot ( 2^{2k(n/2 + k)} )^{2^{n/2 -k}},$$
since the constant $\mu$ in Theorem~\ref{factbound}  is not greater than $1$.
Using the Stirling's approximation~(\ref{stirling}), we see that
\begin{gather*}
\log_2 \widetilde{N}_{n_2}^{2k} \leq  (n/2 - k) \cdot 2^{n/2 - k} +2 k (n/2 +k) \cdot 2^{n/2 -k}  + o(n2^{n/2}) =  \\
 \frac{2k+1}{2^{k+1}} \cdot n  2^{n/2} + o(n 2^{n/2}).
 \end{gather*}
This number is maximal when $k=1$.
\end{proof}

For a more detailed study of  the number of partitions of $F_2^{n}$ into $k$-dimensional affine subspaces see~\cite{BakTar.affpart}.

\section{Bounds on the number of $2$-spreads} \label{2spreadbound}

A collection $\mathcal{S} = \{ S_1, \ldots, S_M \}$ of $d$-dimensional linear spaces $S_i \subseteq F_2^n$, where $M = \frac{2^n-1}{2^d-1}$, is called a \textit{(full) $d$-spread} if each $x \in F_2^n \setminus \{ \overline{0}\}$ belongs to exactly one $S_i$. Equivalently, a $d$-spread is a partition of $F_n^2 \setminus \{ \overline{0} \}$ into $S_i \setminus \{ \overline{0}\}$, where $S_i$ are $d$-dimensional linear subspaces. Although spreads are interesting in themselves, they are often used  as components of  other combinatorial designs. For example,  the ${\cal PS}$ construction~\cite{mesnager.bentfunc} of bent functions on $n$ variables is based on partial $(n/2)$-spreads.

Acting similarly to the proof of Theorem~\ref{affindecomp}, we find the asymptotics of the logarithm of the number of  $2$-spreads.

\begin{thm} \label{lineardecomp}
Let $n$ be even and $W_n$ be the  number of unordered partitions of $F_2^n$ into $2$-dimensional linear subspaces (number of $2$-spreads). Then
$$ \log_2 W_n  =  \frac{n}{3} \cdot 2^n + o(n2^n).  $$
\end{thm}

Since the number of subspaces composing a $2$-spread is $M = \frac{2^{n}-1}{3}$, $2$-spreads exist only if $n$ is even.

To prove this theorem, we use a connection between the number of $2$-spreads $\mathcal{S}$ in $F_2^n$  and the number of transversals in specific  latin squares.

Let us denote by $L_n$ the latin square of order $2^n$ correspoding to the Cayley table of the group $\mathbb{Z}_2^n$: the latin square $L_n$ has entries $l_{\alpha_1, \alpha_2} = \alpha_3$, where $\alpha_{i} \in F_2^n$, if and only if $\alpha_1 \oplus \alpha_2 \oplus \alpha_3 = \overline{0}$.

In what follows, we denote entries of $L_m$ by tuples $(\alpha_1, \alpha_2, \alpha_3)$, $\alpha_{i} \in F_2^n$, satisfying $\alpha_1 \oplus \alpha_2 \oplus \alpha_3 = \overline{0}$. A transversal in  the latin square $L_n$ is a collection of $2^n$ entries
$$(\alpha_1^1, \alpha_2^1, \alpha_3^1), \ldots, (\alpha^{2^n}_1, \alpha_2^{2^n}, \alpha_3^{2^n}) $$
such that for each $j=1, \ldots, 3$ all $\alpha^i_j$ are different, $i = 1, \ldots, 2^n$.

From~\cite[Theorem 1.2]{eberhard.moreaddtrip} we have the following asymptotics of the number of transversals in $L_n$.

\begin{thm}[\cite{eberhard.moreaddtrip}] \label{ebersqbound}
The number $T_n$ of transversals in the Cayley table $L_n$ of the group $\mathbb{Z}_2^n$ of order $2^n$ is
$$T_n = (e^{-1/2} + o(1)) \frac{2^n!^2}{2^{n  \cdot ( 2^n -1)}}$$
as $n \rightarrow \infty$.
\end{thm}

\begin{proof}[Proof of Theorem~\ref{lineardecomp}]
The proof of the lower bound is based on transversals in the latin square $L_{n-2}$. For shortness, denote the order of  $L_{n-2}$ by $m = 2^{n-2}$.

Let a collection $R$ of $m$ entries
$$(\alpha_1^1, \alpha_2^1, \alpha_3^1), \ldots, (\alpha^m_1, \alpha_2^m, \alpha_3^m),$$
$\alpha^i_j \in F_2^{n-2}$, be a transversal in $L_{n-2}$. In particular, $\alpha^i_1 \oplus \alpha_2^i \oplus \alpha_3^i = \overline{0}$ for each $i = 1, \ldots, m$ and for given $j$ all $\alpha^i_j$ are different.

For every such collection $R$, consider a collection $R'$ of $m$ sets
$$\{ \beta_1^1,\beta_2^1 , \beta_3^1\}, \ldots, \{\beta^m_1, \beta_2^m, \beta_3^m\}, $$
where $\beta^i_j \in F_2^n$ and
$$\beta^i_j =\left\{   \begin{array}{l}
(\alpha^i_j, 0, 1) \mbox{ if } j = 1; \\
(\alpha^i_j, 1, 0) \mbox{ if } j = 2; \\
(\alpha^i_j, 1, 1) \mbox{ if } j = 3.
 \end{array}  \right. $$

Note that for each $i = 1, \ldots, m$ and $\{ \beta_1^i,\beta_2^i , \beta_3^i\} \in R'$ the set $\{ \overline{0}, \beta_1^i,\beta_2^i , \beta_3^i \}$ is a $2$-dimensional linear subspace in $F_2^n$, since $\overline{0} \oplus \beta^i_1 \oplus \beta_2^i \oplus \beta_3^i = \overline{0}$. Since $R$ is a transversal, all of these linear subspaces intersect only at $\overline{0}$. Let $R'' = \{\{ \overline{0}, \beta_1^i, \beta_2^i, \beta_3^i\}  : \{ \beta_1^i, \beta_2^i, \beta_3^i\} \in R' \}$ be a collection of such linear subspaces in $F_2^n$.

Subspaces from $R''$ cover all nonzero elements of $F_2^n$ except for elements of the form $(\alpha_j^i, 0,0)$. In order to cover them, consider a  $2$-spread $\mathcal{S} = \{ \{ \overline{0}, \gamma_1^i,  \gamma_2^i , \gamma_3^i \} \}_{i=1}^r$ in $F_2^{n-2}$, $r = \frac{2^{n-2} -1}{3}$.  For each such collection of linear subspaces $\mathcal{S}$ in $F_2^{n-2}$, we construct a collection $\mathcal{S}'$ of linear subspaces $\{ \{ \overline{0}, \delta_1^i,  \delta_2^i , \delta_3^i\} \}_{i=1}^r$ in $F_2^n$, where $\delta^i_j = (\gamma^i_j, 0 ,0)$.

It is easy to see that  $\{ \overline{0}, \delta_1^i,  \delta_2^i , \delta_3^i\} \in S'$ are linear spaces in $F_2^n$ and  all of them intersect only at $\overline{0}$. Moreover, linear subspaces from $\mathcal{S}'$ and $R''$ also intersect only at $\overline{0}$. At last, for each $x \in F_2^n \setminus \{ \overline{0}\}$ there is a linear subspace $S \in R'' \cup \mathcal{S}'$ such that $x \in S$.

Consequently, $R'' \cup \mathcal{S}'$ is a $2$-spread in $F_2^n$ and we have the inequality
$$W_n \geq T_{n-2} \cdot W_{n-2}.$$

Using Theorem~\ref{ebersqbound} and the Stirling's approximation~(\ref{stirling}), we see that
$$ \log_2 T_n \geq n2^n + o(n2^n).$$

Solving the recurrence relation
$$\log_2 W_n \geq \log_2 W_{n-2} + (n-2)\cdot 2^{n-2} + o(n2^n),~~W_2 = 1,$$
we obtain
$$\log_2 W_n \geq  \frac{n}{3} \cdot 2^n + o(n2^n).$$

To bound the number of $2$-spreads from above, we utilize a correspondence between $2$-spreads and perfect matchings in an appropriate hypergraph.

Let $\mathcal{G}_n$ be a $3$-uniform hypergraph with the vertex set $F_2^n \setminus \{\overline{0}\}$ such that $\{ v_1,v_2,v_3 \}$ is a hyperedge in  $\mathcal{G}_n$  if and only if $v_1 \oplus v_2 \oplus v_3 = \overline{0}$. Note that the degree of each vertex in the hypergraph $\mathcal{G}_n$ is $2^{n-1}-1$ and every perfect matching in $\mathcal{G}_n$ is a $2$-spread in $F_2^n$.

Using Theorem~\ref{factbound}, we have that
$$W_n \leq (\mu (2^{n-1} -1))^{\frac{2^n - 1}{3}}. $$

Consequently,
$$\log_2 W_n \leq  \frac{n}{3} \cdot 2^n + o(n2^n).$$
\end{proof}

\section*{Acknowledgements}

The work of V. N. Potapov and A. A. Taranenko was carried out within the framework of the state contract of the Sobolev Institute of Mathematics (project no. FWNF-2022-0017).

\subsection*{Data deposition information}

The paper has no associated data.

\subsection*{Conflict of interests}

There are no competing interests.

\end{document}